\documentclass[12pt]{article}
\usepackage{times}
\usepackage{amsmath,amsfonts,amssymb,amsthm}
\usepackage{calrsfs}

\newtheorem{Theorem}{Theorem}

\renewcommand\P{{\mathbb P}}
\newcommand\E{{\mathbb E}}
\newcommand\R{{\mathbb R}}
\newcommand\I{{\mathbb I}}

\begin{document}
\title{A look at perpetuities via asymptotically
homogeneous in space Markov chains}
\author{Dmitry Korshunov\\ Lancaster University}
\date{}
\maketitle

It is shown how a natural representation of perpetuities 
as asymptotically homogeneous in space Markov chains 
allows to prove various asymptotic tail results for stable 
perpetuities and limit theorems for unstable ones.
Some of these results are new while others essentially
improve moment conditions known in the literature.
Both subexponential and Cram\'er's cases are considered.

60H25, 60J05, 60F10, 60F05;
perpetuity,
stochastic difference equation,
Markov chain,
large deviations,
subexponentiality,
regular variation,
Cram\'er's case,
limit theorems

\section{Introduction}

Let $(\xi,\eta)$ be a random vector in $\R^2$ 
such that $\P\{\xi>0\}>0$.
Let $(\xi_k,\eta_k)$, $k\ge 1$, be independent copies of $(\xi,\eta)$.
Denote $S_0:=0$ and $S_n:=\xi_1+\ldots+\xi_n$.
If $\E\xi=-a\in(-\infty,0)$ then by the strong law of large numbers,
with probability $1$, $-2an\le S_n\le -an/2$
ultimately in $n$, so the process
\begin{eqnarray}\label{D}
D_n &:=& \sum_{k=1}^n \eta_k e^{S_k},\quad n\ge 1,
\end{eqnarray}
is stochastically bounded if and only if $\E\log(1+|\eta|)<\infty$.
If so, then the perpetuity
\begin{eqnarray*}
D_\infty &:=& \sum_{k=1}^\infty \eta_k e^{S_k}
\end{eqnarray*}
is finite with probability $1$ and $D_n\stackrel{a.s.}\to D_\infty$. 
Stability results for more general $D_n$ are dealt with in Vervaat \cite{Vervaat}; 
the case where $\E\xi$ is not necessarily finite is treated 
by Goldie and Maller \cite{GoldieMaller}.

In this paper we are interested in the tail asymptotic behaviour 
of the distributions of $D_n$ and $D_\infty$.
We also consider a Markov modulated perpetuity $D_n$ and maxima
\begin{eqnarray}\label{M.n}
M_n &:=& \max_{k\le n}D_k.
\end{eqnarray}

In the context of random difference equations $R_n=B_n+A_nR_{n-1}$
with positive $A_n=e^{\xi_n}$ and $B_n=\eta_n$, the processes
\begin{eqnarray}\label{widetilde.D}
\widetilde D_n &:=& \sum_{k=1}^n \eta_k e^{S_{k-1}},
\quad \widetilde M_n\ :=\ \max_{k\le n}\widetilde D_k
\end{eqnarray}
are of interest; in particular, the distribution of $\widetilde D_\infty$
represents the unique stationary distribution of the chain $R_n$
which exists provided $\E\xi<0$ and $\E\log(1+|\eta|)<\infty$.

See e.g. Kesten \cite[Theorem 5]{Kesten1973}, 
Goldie \cite[Theorem 4.1]{Goldie1991} 
where power tail asymptotics for $\widetilde D_\infty$ is proven
under the `Cram\'er condition' $\E e^{\beta\xi}=1$ for some $\beta>0$;
see also Dyszewski \cite{Dyszewski} where the case of subexponential 
$\xi$ is considered. Some other cases are considered in Grey \cite{Grey},
Konstantinides and Mikosch \cite{KM}, Goldie and Gr\"ubel \cite{GG}.

The process $\widetilde D_n$ may be constructed as $D_n$ 
with reference vector $(\xi,\eta e^{-\xi})$ instead of $(\xi,\eta)$.
And vice versa, the process $D_n$ is the same as $\widetilde D_n$ 
with reference vector $(\xi,\eta e^\xi)$.
This allows to translate results obtained for $D_n$ 
or $\widetilde D_n$ to each other.

Both perpetuities and stochastic difference equations have many
important applications, among them life insurance and finance,
nuclear technology, sociology, 
random walks and branching processes in random environments, 
extreme-value analysis, one-dimensional ARCH processes, etc. 
For particularities, we refer the reader to, for instance, 
Embrechts and Goldie \cite{EG}, Rachev and Samorodnitsky \cite{RS} 
and Vervaat \cite{Vervaat} for a comprehensive survey of the literature.

By the definition of $D_n$,
\begin{eqnarray}\label{eq:D.D.prime}
D_{n+1} &=& \eta_1e^{\xi_1}+e^{\xi_1}(\eta_2e^{\xi_2}
+\ldots+\eta_{n+1}e^{\xi_2+\ldots+\xi_{n+1}})\nonumber\\
&=:& \eta_1e^{\xi_1}+e^{\xi_1} D'_n,
\end{eqnarray}
where
$$
D'_n:=\eta_2e^{\xi_2}+\ldots+\eta_{n+1}e^{\xi_2+\ldots+\xi_{n+1}}
=_{\rm st} D_n
$$
and $D_n'$ does not depend on $(\xi_1,\eta_1)$. Define
$$
Y_n:=f(D_n)\in\R\quad\mbox{ and }\quad Y_n':=f(D_n')\in\R.
$$
where
\begin{eqnarray}\label{eq:f}
f(x) &:=& \left\{
\begin{array}{ll}
\log x&\mbox{for }x\ge e,\\
-\log |x|&\mbox{for }x\le -e,\\
x/e&\mbox{for }x\in[-e,e],
\end{array}
\right.
\end{eqnarray}
so $f(x):\R\to\R$ is a continuous increasing function
such that $f(x)\ge \log x$ for all $x>0$ and $f(x)\le -\log|x|$ for all $x<0$.
Also define the following random field
\begin{eqnarray}\label{jump.ge.0}
\xi(y) &:=& f(\eta e^\xi+e^\xi f^{-1}(y))-y\nonumber\\
&=& f(\eta e^\xi+e^\xi e^y)-y\quad\mbox{if }y>1,\nonumber\\
&=& \xi+\log(1+\eta e^{-y})
\quad\mbox{if } y>1\mbox{ and }e^\xi(\eta+e^y)>e.
\end{eqnarray}
Then the recursion \eqref{eq:D.D.prime} may be rewritten as
\begin{eqnarray*}
Y_{n+1} &=_d& Y'_n+\xi(Y'_n).
\end{eqnarray*}

Let $\{\xi_n(y),y\in\R\}$, $n\ge 1$, be independent copies 
of the random field $\xi(y)$.
The sequence $Y_n$ is not a Markov chain, while the sequence 
$X_n$ defined by the equalities $X_1:=f(\eta_1e^{\xi_1})$ and
\begin{eqnarray*}
X_{n+1} &:=& X_n+\xi_{n+1}(X_n)
\end{eqnarray*}
is a time-homogeneous Markov chain on $\R$
due to independence of $\xi_n(y)$, $n\ge 1$;
call it the {\it associated} Markov chain.
Despite the fact that the distributions of the
sequences $\{X_n\}$ and $\{Y_n\}$ are different,
$$
X_n=_{\rm st} Y_n\quad\mbox{ for every fixed }n,
$$
which allows to compute the distribution tail of the perpetuity $D_n$ 
via the distribution tail of the Markov chain $X_n$,
\begin{eqnarray}\label{D.via.X}
\P\{D_n>x\} &=& \P\{X_n>\log x\}\quad\mbox{for }x>e.
\end{eqnarray}
In particular, the distribution of $D_\infty$ coincides on the set $[e,\infty)$
with the invariant distribution of the Markov chain $e^{X_n}$.
Here the situation is similar to that for the maximum os sums
$\max\{0,S_1,\ldots,S_n\}$;
it is not a Markov chain but coincides in distribution
with a Markov chain $W_n$ where
\begin{equation}\label{W}
W_0:=0\ \mbox{ and }\
W_n:=(W_{n-1}+\xi_n)^+
\end{equation}
(see, for example, Feller \cite[Chap. VI, section 9]{Feller}).

A Markov chain $X_n$ is called 
{\it asymptotically homogeneous in space} if the distribution 
of its jump $\xi_n(x)$ weakly converges as $x\to\infty$.
(A similar notion of additive Markov process was introduced 
by Aldous in \cite[Section C11]{Aldous} where stronger convergence 
of $\xi(x)$---in total variation---is assumed.)

The associated Markov chain $X_n$ is asymptotically 
homogeneous in space with limiting jump $\xi$; 
it is particularly emphasised by Goldie in \cite[Section 2]{Goldie1991}.
Let us underline that, in general, $\xi+\log(1+\eta e^{-x})$ 
may not converge to $\xi$ as $x\to\infty$ in total variation norm.

In the literature, some other random equations are also considered,
for example (see e.g. Goldie \cite{Goldie1991})
\begin{eqnarray*}
R_{n+1} &:=& \max(\eta_{n+1},\ e^{\xi_{n+1}}R_n),
\end{eqnarray*}
where the associated Markov chain has jumps
$\xi(x)=\max(\xi,\ \log\eta-x)$ eventually in $x$
and so it is again asymptotically space-homogeneous.

Asymptotically homogeneous in space
Markov chains were studied in \cite{BK2001,BK2002,K2004}
from the point of view of the asymptotic behaviour
of the probabilities of large deviations.
In particular, it was shown there that tail asymptotics
of the invariant measure heavily depends on the
tail properties of the limiting jump $\xi$ of the chain.
Let us recall some relevant notions.

Denote exponential moments of $\xi$ by 
$\varphi(\lambda)=\E e^{\lambda\xi}$ and consider 
$$
\beta=\sup\{\lambda\ge 0:\varphi(\lambda)\le 1\}.
$$
Since $\P\{\xi>0\}>0$, $\beta<\infty$.
In this paper we study the following two basic cases:
\begin{enumerate}
\item[(i)] $\beta=0$, the heavy-tailed case where all positive 
exponential moments of $\xi$ are infinite;
\item[(ii)] $\beta>0$ and $\varphi(\beta)=1$, the Cram\'er case.
\end{enumerate}
As clear from the theory of random walks, asymptotic behaviour 
of $\P\{X_n>x\}$ for the associated Markov chain $X_n$ 
should be very different in these two cases.

The paper is organised as follows. 
In Section \ref{sec:subexp} we start with subexponential 
asymptotics for $D_n$ and $\widetilde D_n$. 
Then in Section \ref{sec:Cramer} we study perpetuities 
in the Cram\'er case while in Section \ref{sec:Markov} 
we do it for Markov modulated perpetuities.
The last Section \ref{sec:clt} is devoted to limit theorems
for transient perpetuities.

\section{Subexponential asymptotics}
\label{sec:subexp}

In this section we consider the case where $\xi$ is heavy-tailed
and $\eta>0$, so $D_n>0$ and $\widetilde D_n>0$.
We show, in particular, that the most probable way by which large values 
of both $D_n$ and $\widetilde D_n$ do occur is a single big jump;
this principle is well known in the theory of subexponential distributions,
see e.g. \cite[Theorem 5.4]{FKZ}.
Let us first recall relevant distribution classes.
We denote by $\overline H(x)=H(x,\infty)$ the tail of a distribution $H$.

For a distribution $H$ with finite expectation, 
we define the {\it integrated tail distribution} $H_I$ on $\R^+$ by its tail:
$$
\overline H_I(x) :=
\min\Bigl(1,\int_x^\infty \overline H(y)dy\Bigr).
$$

A distribution $H$ with right unbounded support 
is called to be {\it long-tailed} if, for each fixed $y$,
$\overline H(x+y)\sim\overline H(x)$ as $x\to\infty$.

A distribution $H$ on $\R^+$
with unbounded support is called to be {\it subexponential}
if $\overline{H*H}(x)\sim 2\overline H(x)$ as $x\to\infty$.
Equivalently,
$\P\{\zeta_1+\zeta_2>x\}\sim 2\P\{\zeta_1>x\}$,
where random variables $\zeta_1$ and $\zeta_2$
are independent with distribution $H$.
A distribution $H$ of a random variable $\zeta$ on $\R$ 
with right-unbounded support is called to be {\it subexponential} 
if the distribution of $\zeta^+$ is so.
Standard examples of subexponential distributions are given by
Pareto, regularly varying, log-normal, 
Weibull with shape parameter $k<1$ distributions.

As well known (see, e.g. \cite[Lemma 3.2]{FKZ}) 
subexponentiality of $H$ on $\R^+$ implies that $H$ is long-tailed.
In~particular, if the distribution of a random variable 
$\zeta\ge 0$ is subexponential then $\zeta$ is heavy-tailed.

A distribution $H$ with right unbounded support
and finite expectation is called to be {\it strong subexponential} if 
$$
\int_0^x\overline H(y)\overline H(x-y)dy
\ \sim\ 2\overline H(x)\int_0^\infty\overline H(y)dy
\quad\mbox{as }x\to\infty.
$$
Strong subexponentiality of $H$ implies that both 
$H$ and $H_I$ are subexponential, see e.g. \cite[Theorem 3.27]{FKZ}.
Standard subexponential distributions are usually strong subexponential too.

For the perpetuity $D_n$, 
denote the distribution of the sum $\xi+\log(1+\eta)$ by $H$.

For the perpetuity $\widetilde D_n$, 
denote the distribution of the $\max(\xi,\log(1+\eta))$ by $\widetilde H$.

\begin{Theorem}\label{thm:subexp}
Suppose that $\E\xi=-a<0$, $\eta>0$ and $\E\log(1+\eta)<\infty$,
so that $D_n$ is a convergent perpetuity.
If the integrated tail distribution
$H_I$ is long-tailed, then
$$
\liminf_{x\to\infty}
\frac{\P\{D_\infty>x\}}{\overline{H_I}(\log x)} \ge \frac{1}{a}.
$$
If, in addition, the integrated tail distribution $H_I$
is subexponential then
\begin{eqnarray*}
\P\{D_\infty>x\} &\sim& \frac{1}{a}\overline H_I(\log x)
\quad\mbox{as }x\to\infty.
\end{eqnarray*}

The same results hold for $\widetilde D_\infty$ 
if the distribution $H_I$ is replaced by $\widetilde H_I$.
\end{Theorem}

The tail asymptotics for $D_\infty$ and $\widetilde D_\infty$ are determined 
by the distributions $H$ of $\xi+\log(1+\eta)$ and $\widetilde H$ of 
$\max(\xi,\log(1+\eta))$ respectively. If $\xi$ and $\eta$ are independent, 
then---for most standard subexponential distributions of 
$\xi$ and $\eta$---$H$ and $\widetilde H$ are tail equivalent, 
hence tail asymptotics for $D_\infty$ and $\widetilde D_\infty$
are asymptotically equivalent. The situation becomes different 
for dependent $\xi$ and $\eta$. 
For example, if $\xi=\log(1+\eta)+{\rm const}$ then
$\overline H(x)\sim \overline{\widetilde H}(x/2)$, so the tail of $H$ is
heavier than that of $\widetilde H$.

Denote the distribution of $\xi$ by $F$ and
the distribution of $\log(1+\eta)\ge0$ by $G$.
Notice that $H_I$ (and $\widetilde H_I$) is automatically subexponential
if $\xi_1$ and $\eta_1$ are independent, $F_I$ is subexponential 
and $\overline G_I(x)=o(\overline F_I(x))$ as $x\to\infty$
(see, e.g. Corollary 3.18 in \cite{FKZ}).

The tail asymptotics for $\widetilde D_\infty$ was proven
by Dyszewski in \cite[Theorem 3.1]{Dyszewski} under additional 
assumption that moment of order $1+\gamma$ of $\widetilde H$ if finite.
If $\eta$ takes values of both signs, then only asymptotic 
upper and lower bounds are known for the tail of the perpetuity, 
see Dyszewski \cite[Theorem 3.1]{Dyszewski};
we do not cover this case in the present paper.

The last theorem seems to be deducible from \cite[Theorem 3]{BK2002}
where subexponential asymptotics were proven for asymptotically
homogeneous in space Markov chains. 
But, first, it is formally assumed in \cite[Theorem 3]{BK2002}
that the distribution of a Markov chain $X_n$ converges to
the invariant distribution in total variation norm
which is not always the case for perpetuities.
Second,  perpetuity possesses some specific properties 
which allow to prove asymptotics in a more simple way 
than it is done in \cite[Theorem 3]{BK2002}; 
we present such a proof below but still it follows some ideas 
of the proof for Markov chains in \cite{BK2002}.

\begin{proof}[Proof of Theorem \ref{thm:subexp}]
In the case where $\eta>0$ and so $D_n>0$, 
it is more convenient to consider the logarithm of $D_n$, 
so that the associated Markov chain $X_n$ has jumps
$$
\xi(x) \ =\ \xi+\log(1+\eta e^{-x}),\quad x\in\R.
$$

First consider $D_\infty$.
Since $\eta>0$, the family of jumps $\xi(y)$, $y\ge 0$, 
possesses an integrable minorant
\begin{eqnarray}\label{xi.x.min}
\xi(y) &\ge_{st}& \xi.
\end{eqnarray}
Fix $\varepsilon>0$. The family of random variables $\log(1+\eta e^{-y})$, 
$y\ge 0$, possesses an integrable majorant $\log(1+\eta)$ and 
$\log(1+\eta e^{-y})\to 0$ as $y\to\infty$ in probability.
Then it follows by the dominated convergence theorem that, 
for some sufficiently large $x_1$,
\begin{eqnarray}\label{choice.x1}
\E\log(1+\eta e^{-x_1}) &\le& \varepsilon.
\end{eqnarray}
Therefore, the family of jumps $\xi(y)$, $y\ge x_1$, 
possesses an integrable majorant
\begin{eqnarray}\label{xi.x.maj}
\xi(y) &\le_{st}& \xi+\log(1+\eta e^{-x_1}).
\end{eqnarray}

Since $D_n$ is assumed to be convergent,
the associated Markov chain $X_n$ is stable, 
so there exists an $x_2>0$ such that
$$
\P\{X_n\in[-x_2,x_2]\}\ \ge\ 1-\varepsilon \quad\mbox{for all }n\ge 0.
$$
For all $k$, $n$ and $A$ consider the event
$$
B(k,n,A):=\{\xi_{k+1}+\ldots+\xi_{k+j}
\ge -A-n(a+\varepsilon)\mbox{ for all }j\le n\}.
$$
By the strong law of large numbers,
there exists a sufficiently large $A$ such that
$$
\P\{B(k,n,A)\}\ \ge\ 1-\varepsilon\quad\mbox{for all }k\mbox{ and }n.
$$
It follows from \eqref{xi.x.min} that any of the events
$$
\{X_{k-1}\in[-x_2,x_2],\ 
\xi_k+\log(\eta_k+e^{-x_2})-x_2>x+A+(n-k)(a+\varepsilon),\ B(k,n-k,A)\}
$$
implies $X_n>x$ and they are pairwise disjoint.
Taking into account the inequality
\begin{eqnarray*}
\log(\eta_k+e^{-x_2}) &=& \log(1+\eta_k)
-\log\frac{\eta_k+1}{\eta_k+e^{-x_2}}\\
&\ge& \log(1+\eta_k)-x_2,
\end{eqnarray*}
we obtain, by the Markov property,
\begin{eqnarray*}
\P\{X_n>x\} &\ge& \sum_{k=1}^{n-1}\P\{X_{k-1}\in[-x_2,x_2]\}
\overline H(x+2x_2+A+(n-k)(a+\varepsilon))\\
&&\hspace{60mm}\times \P\{B(k,n-k,A)\}\\
&\ge& (1-\varepsilon)^2\sum_{k=1}^{n-1}
\overline H(x+2x_2+A+(n-k)(a+\varepsilon)).
\end{eqnarray*}
Since the tail is a non-increasing function, the last sum is not less than
$$
\frac{1}{a+\varepsilon}
\int_{a+\varepsilon}^{n(a+\varepsilon)}
\overline H(x+2x_2+A+y)dy.
$$
Letting $n\to\infty$ we obtain that the tail
at point $x$ of the stationary distribution of the
associated Markov chain $X$ is not less than
$$
\frac{(1-\varepsilon)^2}{a+\varepsilon}
\int_{a+\varepsilon}^\infty \overline H(x+2x_1+A+y)dy
=\frac{(1-\varepsilon)^2}{a+\varepsilon}
\overline{H_I}(x+2x_2+A+a+\varepsilon).
$$
Since the integrated tail distribution $H_I$ 
is assumed to be long-tailed, 
$$
\overline{H_I}(x+2x_2+A+a+\varepsilon)\sim
\overline{H_I}(x)\quad\mbox{as }x\to\infty.
$$
Summarising altogether we deduce that,
for every fixed $\varepsilon>0$,
$$
\liminf_{x\to\infty}
\frac{\P\{D_\infty>x\}}{\overline{H_I}(\log x)}
\ge \frac{(1-\varepsilon)^2}{a+\varepsilon},
$$
which implies the lower bound of the theorem due
to the arbitrary choice of $\varepsilon>0$.

Now turn to the asymptotic upper bound under assumption that
the integrated tail distribution $H_I$ is subexponential.
Fix $\varepsilon\in(0,a)$. Let $x_1$ be defined as in \eqref{choice.x1},
so $\E\xi(x_1)\le -a+\varepsilon$. Take
$$
\zeta_n:=\xi_n+\log(1+\eta_n e^{-x_1})
$$
and let $H_1$ be its distribution. Since
$$
\xi+\log(1+\eta)-x_1\ \le\ \zeta\ \le\ \xi+\log(1+\eta),
$$
we have
$\overline H(x+x_1)\le\overline H_1(x)\le\overline H(x)$.
Then subexponentiality of $H_I$ yields subexponentiality 
of the integrated tail distribution of $H_1$ and 
$\overline H_{1,I}(x)\sim\overline H_I(x)$ as $x\to\infty$.

The jumps $\xi_n(x)$ of the chain $X_n$ possess the upper bound
\eqref{xi.x.maj} which may be rewritten as
\begin{equation}\label{xi.eta.1}
\xi_n(x) \le \zeta_n\quad\mbox{for all }x\ge x_1.
\end{equation}
In addition, by the inequality
\begin{eqnarray}\label{log.uv}
\log(1+uv) &\le& \log(1+u)+\log v\quad\mbox{for }u\ge 0,\ v\ge 1,
\end{eqnarray}
we have
\begin{eqnarray}\label{xi.eta.2}
x+\xi_n(x) &=& x+\xi_n+\log(1+\eta_ne^{-x})\nonumber\\
&\le& x+\xi_n+\log(1+\eta_ne^{-x_1})+\log e^{x_1-x}\nonumber\\
&=& \zeta_n+x_1
\quad\mbox{ for all }x\le x_1.
\end{eqnarray}
Consider a random walk $Z_n$ with delay at the origin with jumps $\zeta$'s:
$$
Z_0:=0,\ \ Z_n:=(Z_{n-1}+\zeta_n)^+.
$$
The upper bounds \eqref{xi.eta.1} and \eqref{xi.eta.2} yield that
$$
X_n\le x_1+Z_n\quad\mbox{ for all }n.
$$
so that $X_n$ is dominated by the random walk on $[x_1,\infty)$ 
with delay at point $x_1$.
Since the integrated tail distribution $H_{1,I}$ is assumed
to be subexponential, the tail of the invariant measure of 
the chain $Z_n$ is asymptotically equivalent 
to $\overline H_{1,I}(x)/(a-\varepsilon)\sim\overline H_I(x)/(a-\varepsilon)$ 
as $x\to\infty$, see, for example, \cite[Theorem 5.2]{FKZ}. 
Thus, the tail of the invariant measure of $X_n$ is asymptotically 
not greater than $\overline H_I(x-x_1)/(a-\varepsilon)$
which is equivalent to $\overline H_I(x)/(a-\varepsilon)$,
since $H_I$ is long-tailed by the subexponentiality.
Hence,
$$
\limsup_{x\to\infty} \frac{\P\{D_\infty>x\}}{\overline H_I(\log x)}
\le \frac{1}{a-\varepsilon}.
$$
By the arbitrary choice of $\varepsilon>0$
together with the lower bound proven above this completes the proof
of the second theorem assertion for $D_\infty$.

The result for $\widetilde D_\infty$ is immediate if we prove that
$\widetilde H_I$ is subexponential if and only if 
the integrated tail distribution of $\log(e^\xi+\eta)$ is so.
Indeed, since
$$
\log(e^\xi+1+\eta)-\log 2\ \le\ \log(e^\xi+\eta)\ \le\ \log(e^\xi+1+\eta)
$$
on the event $e^\xi+\eta\ge 1$, subexponentiality of 
the integrated tail distribution of $\log(e^\xi+\eta)$
is equivalent to subexponentiality of that for $\log(e^\xi+1+\eta)$.
Then inequalities
\begin{eqnarray*}
\log(e^\xi+1+\eta) &\le& \log\bigl(2\max(e^\xi,1+\eta)\bigr)
\ =\ \log 2+\max\bigl(\xi,\log(1+\eta)\bigr)
\end{eqnarray*}
and
\begin{eqnarray*}
\log(e^\xi+1+\eta) &\ge& \log\bigl(\max(e^\xi,1+\eta)\bigr)
\ =\ \max\bigl(\xi,\log(1+\eta)\bigr)
\end{eqnarray*}
imply the required conclusion.
\end{proof}

The same arguments with the same minorants and majorants
allow us to conclude the following result for the finite
time horizon asymptotics if we apply Theorem 5.3
from \cite{FKZ} instead of Theorems 5.1 and 5.2.

\begin{Theorem}\label{subexp.n}
Suppose that $\E\xi=-a<0$, $\eta>0$ and $\E\log(1+\eta)<\infty$.
If the distribution $H$ of $\xi+\log(1+\eta)$
is subexponential then, for each fixed $n\ge 1$,
\begin{eqnarray}\label{n.asy.D}
\P\{D_n>x\} &\sim& \frac{1}{a}
\int_{\log x}^{\log x+na}\overline H(y)dy
\quad\mbox{ as }x\to\infty.
\end{eqnarray}
If $H$ is strong subexponential then \eqref{n.asy.D} holds uniformly in $n\ge1$.

The same results hold for $\widetilde D_n$ 
if the distribution $H$ is replaced by $\widetilde H$.
\end{Theorem}

The main contribution of this theorem is assertion stating
uniformity in $n\ge 1$. The simple part stating \eqref{n.asy.D}
for $\widetilde D_n$ for a fixed $n$ is proven by Dyszewski
in \cite[Theorem 3.3]{Dyszewski}.

We conclude this section by a version of the principle of a single big
jump for $D_\infty$, $D_n$, $\widetilde D_\infty$, and $\widetilde D_n$. 
For simplicity we consider the case 
where $\eta\ge\delta$ for some constant $\delta>0$. Then
\begin{eqnarray}\label{log.eta.delta}
\log(1+\eta) &\le& \log\eta+\log(1+1/\delta).
\end{eqnarray}
For any $C>0$ and $\varepsilon>0$ consider events
\begin{eqnarray*}
B_k &:=& \bigl\{|S_j+aj|\le (j\varepsilon+C)/2
\mbox{ and }\log\eta_j\le (j\varepsilon+C)/2\mbox{ for all }j\le k,\\
&&\hspace{65mm} \xi_{k+1}+\log\eta_{k+1}>\log x+ka\bigr\}
\end{eqnarray*}
which, for large $x$, roughly speaking means that
up to time $k$ the random walk $S_j$ moves down
according to the strong law of large numbers and
then a big value of $\xi_{k+1}+\log\eta_{k+1}$ occurs for some $k$. 
As stated in the next theorem, the union of these events
describes the most probable way by which large
deviations of $D_\infty$ and $D_n$ can occur.

For $\widetilde D_n$, we consider events
\begin{eqnarray*}
\widetilde B_k &:=& \bigl\{|S_j+aj|\le (j\varepsilon+C)/2,\
\max(\xi_j,\log\eta_j)\le (j\varepsilon+C)/2\mbox{ for all }j\le k,\\
&&\hspace{58mm} \max(\xi_{k+1},\log\eta_{k+1})>\log x+ka\bigr\}.
\end{eqnarray*}

\begin{Theorem}\label{th:psbj}
Let $H_I$ be subexponential.
Then, for any fixed $\varepsilon>0$,
\begin{eqnarray*}
\lim_{C\to\infty}\lim_{x\to\infty}
\P\{\cup_{k=0}^\infty B_k\mid D_\infty>x\} &=& 1.
\end{eqnarray*}
If, in addition, $H$ is strong subexponential,
then, for any fixed $\varepsilon>0$,
\begin{eqnarray*}
\lim_{C\to\infty}\lim_{x\to\infty}
\inf_{n\ge 1}\P\{\cup_{k=0}^{n-1} B_k\mid D_n>x\} &=& 1.
\end{eqnarray*}

The same results hold for $\widetilde D_\infty$ and $\widetilde D_n$
if the distribution $H$ is replaced by $\widetilde H$ 
and events $B_k$ by $\widetilde B_k$.
\end{Theorem}

\begin{proof}
We prove the assertion for $D_n$ only, 
because the proof for $D_\infty$ is similar.

Since each of the events
\begin{eqnarray*}
B_k^* &:=& \bigl\{|S_j+aj|\le (j\varepsilon+C)/2\mbox{ and }
\log\eta_j\le (j\varepsilon+C)/2\mbox{ for all }j\le k,\\
&&\hspace{25mm} \xi_{k+1}+\log\eta_{k+1}>\log x+C+k(a+\varepsilon)\bigr\},
\quad k\le n-1,
\end{eqnarray*}
is contained in $B_k$ and implies that 
\begin{eqnarray*}
\eta_{k+1}e^{S_{k+1}} &=& e^{S_k+\xi_{k+1}+\log\eta_{k+1}}\ \ge\ e^{\log x},
\end{eqnarray*}
so that $D_n>x$, we consequently have that
\begin{eqnarray}\label{uni.B.M.ge}
\P\{\cup_{k=0}^n B_k\mid D_n>x\}
&\ge& \P\{\cup_{k=0}^{n-1}B_k^*\mid D_n>x\}
= \frac{\P\{\cup_{k=0}^{n-1}B_k^*\}}
{\P\{D_n>x\}}.
\end{eqnarray}
The events $B_k^*$ are disjoint for all $x>e^C$ because
$S_j+\log\eta_j\le (-a+\varepsilon)j+C$ for $j\le k$ on $B_k^*$ while 
\begin{eqnarray*}
S_{k+1}+\log\eta_{k+1} &=& S_k+\xi_{k+1}+\log\eta_{k+1}\\
&\ge& (-a-\varepsilon/2)k-C/2+\log x+C+k(a+\varepsilon)\\
&\ge& \log x,
\end{eqnarray*}
hence
\begin{eqnarray*}
\P\{\cup_{k=0}^{n-1}B_k^*\}
&=& \sum_{k=0}^{n-1}\P\{B_k^*\}\\
&=& \sum_{k=0}^{n-1}\P\Bigl\{|S_j+aj|\le \frac{j\varepsilon+C}{2},
\ \log\eta_j\le \frac{j\varepsilon+C}{2}\mbox{ for all }j\le k\Bigr\}\\
&&\hspace{25mm}\times \P\bigr\{\xi_{k+1}+\log\eta_{k+1}>\log x+C+k(a+\varepsilon)\bigr\}.
\end{eqnarray*}
For any fixed $\gamma>0$, there exists $C$ such that, 
\begin{eqnarray*}
\P\Bigl\{|S_j+aj|\le \frac{j\varepsilon+C}{2},
\ \log\eta_j\le \frac{j\varepsilon+C}{2}\mbox{ for all }j\ge 1\Bigr\}
&\ge& 1-\gamma.
\end{eqnarray*}
Then, for all $x>e^C$,
\begin{eqnarray*}
\P\{\cup_{k=1}^{n-1}B_k^*\}
&\ge& (1-\gamma) \sum_{k=0}^{n-1}
\P\bigl\{\xi+\log\eta>\log x+C+k(a+\varepsilon)\bigr\}.
\end{eqnarray*}
Applying \eqref{log.eta.delta} we get
\begin{eqnarray*}
\P\{\cup_{k=1}^{n-1}B_k^*\}
&\ge& (1-\gamma) \sum_{k=0}^{n-1}
\P\bigl\{\xi+\log(1+\eta)>\log x+C+k(a+\varepsilon)+\log(1+1/\delta)\bigr\}\\
&\sim& (1-\gamma) \sum_{k=0}^{n-1}
\overline H(\log x+k(a+\varepsilon))\\
&\ge& \frac{1-\gamma}{a+\varepsilon}
\int_{\log x}^{\log x+n(a+\varepsilon)}\overline H(y)dy.
\end{eqnarray*}
Substituting this estimate and the asymptotics
for $D_n$ into \eqref{uni.B.M.ge} we deduce that
\begin{eqnarray*}
\lim_{x\to\infty}\inf_{n\ge 1}
\P\{\cup_{k=0}^n B_k\mid D_n>x\}
&\ge& \frac{(1-2\gamma)a}{a+\varepsilon}.
\end{eqnarray*}
Now we can make $\gamma>0$ as small as we please
by choosing a sufficiently large $C$. Therefore,
\begin{eqnarray*}
\lim_{C\to\infty}\lim_{x\to\infty}
\inf_{n\ge 1}\P\{\cup_{k=0}^{n-1} B_k\mid D_n>x\}
&\ge& \frac{a}{a+\varepsilon}.
\end{eqnarray*}
Here the probability on the left is decreasing as
$\varepsilon\downarrow 0$ while the ratio on the right
can be made as close to $1$ as we please by choosing
a sufficiently small $\varepsilon>0$.
This yields that the limit is equal to $1$ for every $\varepsilon>0$. 
\end{proof}




\section{Cram\'er's case}
\label{sec:Cramer}

In this section we consider light-tailed case where $\xi$
possesses some positive exponential moments finite;
the Cram\'er case is studied.
To obtain tail results for $D_\infty$ and $D_n$ in the Cram\'er case
we first recall the corresponding theorem for
asymptotically space-homogeneous Markov chain.
So, let $X_n$ be a Markov chain on $\R$ with jumps $\xi(x)$ 
which weakly converge to $\xi$ as $x\to\infty$; 
let the distribution $F$ of the random variable $\xi$ be non-lattice.
Let $\pi$ be the invariant distribution of $X_n$.

As above, the parameter $\beta>0$ is a positive solution
to the equation $\varphi(\beta)=\E e^{\beta\xi}=1$.
Then the measure $F^{(\beta)}$ defined by the equality
$$
F^{(\beta)}(du)=e^{\beta u}F(du),
$$
is probabilistic. Let $\xi^{(\beta)}$ be a random
variable with distribution $F^{(\beta)}$. Assume that
\begin{eqnarray*}
\alpha  &\equiv& \E\xi^{(\beta)}=\varphi'(\beta)\in (0,\infty).
\end{eqnarray*}

\begin{Theorem}[\cite{K2004}]\label{thm:m.c.cramer}
Let
\begin{eqnarray}\label{cond.3}
\int_{-\infty}^\infty e^{\beta y}|\P\{\xi(x)>y\}-\P\{\xi>y\}|dy
&\le& \delta(x),\quad x\in\R,
\end{eqnarray}
for some bounded decreasing integrable at infinity regularly varying 
function $\delta(x)$. Then
$$
\pi(x,\infty)\ =\ (c+o(1))e^{-\beta x}\quad\mbox{as }x\to\infty,
$$
where
\begin{eqnarray}\label{value.of.c}
c &=& \frac1{\beta\alpha}
\int_{-\infty}^\infty(\E e^{\beta\xi(y)}-1)e^{\beta y}\,\pi(dy)\in[0,\infty).
\end{eqnarray}
If $\E e^{\beta\xi(y)} \ge 1-\gamma(y)$ for some decreasing 
function $\gamma(y)=o(1/y)$ such that $y\gamma(y)$
is integrable at infinity, then $c>0$.

Let in addition $\E e^{\beta X_0}$ be finite,
$\E\xi^2 e^{\beta\xi}<\infty$ and let
the family of jumps $\{\xi(u), u\in\R\}$
possesses a stochastic majorant $\overline\xi$ such that
\begin{eqnarray}\label{maj}
\E \overline\xi^2e^{\beta\overline\xi} &<& \infty.
\end{eqnarray}
Assume also that the chain jumps satisfy the following conditions:
\begin{eqnarray}
\inf_{u\in\R}\E e^{\beta\xi(u)} &>& 0,\label{cond.1}\\
\E \xi(u)e^{\beta\xi(u)}
&=& \alpha+o(1/\sqrt u)\quad\mbox{as }u\to\infty.
\label{cond.2}
\end{eqnarray}
Then the following relation holds:
$$
\P\{X_n>x\}=ce^{-\beta x}
{\mathcal N}_{0,\sigma^2}\bigg(\frac{n\alpha-x}{\sqrt{x/\alpha}}\bigg)
+o(e^{-\beta x})
$$
as $x\to\infty$ uniformly in $n\ge0$, 
where ${\mathcal N}_{0,\sigma^2}$ is the normal cumulative
distribution function with zero expectation and variance 
\begin{eqnarray*}
\sigma^2 &\equiv& {\mathbb Var}\xi^{(\beta)}
= \varphi''(\beta)-(\varphi'(\beta))^2.
\end{eqnarray*}
\end{Theorem}

The last theorem gives a new way for proving the power tail 
asymptotics for the perpetuities $D_\infty$ and 
$\widetilde D_\infty$ in the Cram\'er case. 
We assume a non-lattice distribution of $\xi$.

\begin{Theorem}\label{thm:Cramer}
Suppose that $\E e^{\beta\xi}=1$ and $\alpha:=\E \xi e^{\beta\xi}<\infty$. 
If $\E e^{\beta\xi}|\eta|^\beta<\infty$ then, for some $c>0$,
\begin{eqnarray}\label{D.infty.beta}
\P\{D_\infty>x\} &\sim& \frac{c}{x^\beta}\quad\mbox{ as }x\to\infty.
\end{eqnarray}

The same result holds for $\widetilde D_\infty$ if we assume 
finiteness of $\E |\eta|^\beta$ instead of $\E e^{\beta\xi}|\eta|^\beta$.
\end{Theorem}


Notice that the tail asymptotics for $\widetilde D_\infty$
is due to Kesten \cite[Theorem 5]{Kesten1973};
for a complete proof see Goldie \cite[Theorem 4.1]{Goldie1991}.

\begin{proof}
By Theorem \ref{thm:m.c.cramer}, it is sufficient to check 
that the jumps $\xi(x)$ of the associated Markov chain $X_n$
defined in \eqref{jump.ge.0} satisfy 
\begin{eqnarray}\label{cond.3.ver}
\int_{-\infty}^\infty e^{\beta y}|\P\{\xi(x)>y\}-\P\{\xi>y\}|dy
&=& O(e^{-\delta x})\ \mbox{ as }x\to\infty,
\end{eqnarray}
where $\delta=\min(1,\beta)$. Consider $x>1$. 
Since $f(y)\ge\log y$ for all $y>0$,
\begin{eqnarray}\label{xi.x.below}
\xi(x) &=& f(e^\xi(e^x+\eta))-x\nonumber\\ 
&\ge& \log(e^\xi(e^x+\eta))-x\ =\ \xi+\log(1+\eta e^{-x})
\end{eqnarray}
on the event $\eta>-e^x$. Then $\xi(x)\ge\xi$ on the event $\eta>0$.
On the event $\eta\le 0$,
\begin{eqnarray*}
\xi(x)\ \le\ f(e^{x+\xi})-x &\le& \max(\xi,1-x).
\end{eqnarray*}
Therefore, for $x>1$,
\begin{eqnarray*}
|\P\{\xi(x)>y\}-\P\{\xi>y\}| 
&\le& \P\{\xi(x)>y, \eta>0\}-\P\{\xi>y, \eta>0\}\\
&&\hspace{4mm} +\P\{\xi>y, \eta\le 0\}-\P\{\xi(x)>y, \eta\le 0\}\\
&&\hspace{10mm} +\I\{y\le 1-x\},
\end{eqnarray*}
and hence the integral in \eqref{cond.3.ver} may be bounded by the sum
\begin{eqnarray}\label{decomp.3}
\lefteqn{\hspace{-10mm}\int_{-\infty}^{1-x} e^{\beta y}dy+
\int_{-\infty}^\infty e^{\beta y}
\bigl(\P\{\xi(x)>y, \eta>0\}-\P\{\xi>y, \eta>0\}\bigr)dy}\nonumber\\
&&\hspace{10mm}+\int_{-\infty}^\infty e^{\beta y}
\bigl(\P\{\xi>y, \eta\le 0\}-\P\{\xi(x)>y, \eta\le 0\}\bigr)dy.
\end{eqnarray}
The first integral here is of order $O(e^{-\beta x})$.
The second integral equals
\begin{eqnarray*}
\lefteqn{\beta^{-1}\bigl(\E\{e^{\beta\xi(x)};\ \eta>0\}
-\E\{e^{\beta\xi};\ \eta>0\}\bigr)}\\
&=& \beta^{-1}\bigl(\E\{e^{\beta\xi(x)};\ \eta>0,\xi(x)>1-x\}
-\E\{e^{\beta\xi};\ \eta>0\}\bigr)+O(e^{-\beta x})\\
&=& \beta^{-1}\bigl(\E\{e^{\beta\xi}(1+\eta e^{-x})^\beta;\ \eta>0,\xi(x)>1-x\}
-\E\{e^{\beta\xi};\ \eta>0\}\bigr)+O(e^{-\beta x}).
\end{eqnarray*}
Thus, the second integral is not greater than
\begin{eqnarray*}
\lefteqn{\beta^{-1}\bigl(\E\{e^{\beta\xi}(1+\eta e^{-x})^\beta;\ \eta>0\}
-\E\{e^{\beta\xi};\ \eta>0\}\bigr)+O(e^{-\beta x})}\\
&&\hspace{50mm} =\ \beta^{-1}\E e^{\beta\xi}\bigl((1+\eta^+ e^{-x})^\beta-1\bigr)
+O(e^{-\beta x}).
\end{eqnarray*}
If $\beta\le 1$, then $(1+z)^\beta-1\le z^\beta$ for $z\ge 0$ and hence
\begin{eqnarray*}
\E e^{\beta\xi}\bigl((1+\eta^+ e^{-x})^\beta-1\bigr)
&\le& e^{-\beta x}\E (e^\xi\eta^+)^\beta
\ =\ O(e^{-\beta x})\quad\mbox{as }x\to\infty,
\end{eqnarray*}
since $\E (e^\xi \eta^+)^\beta$ is finite. If $\beta>1$, 
then there is a $c_1$ such that
\begin{eqnarray}\label{z.beta}
(1+z)^\beta-1 &\le& \beta(1+z)^{\beta-1}z\ \le\ c_1(z+z^\beta)
\quad\mbox{for }z\ge 0,
\end{eqnarray}
so 
\begin{eqnarray*}
\E e^{\beta\xi}\bigl((1+\eta^+ e^{-x})^\beta-1\bigr)
&\le& c_1\E e^{\beta\xi}
(\eta^+ e^{-x}+(\eta^+)^\beta e^{-\beta x})
\ =\ O(e^{-x}).
\end{eqnarray*}
Altogether implies that the second integral in
\eqref{decomp.3} is of order $O(e^{-\delta x})$ as $x\to\infty$.

Further, the third integral in \eqref{decomp.3} equals
\begin{eqnarray*}
\beta^{-1}\bigl(\E\{e^{\beta\xi};\ \eta\le 0\}
-\E\{e^{\beta\xi(x)};\ \eta\le 0\}\bigr).
\end{eqnarray*}
In its turn, the difference of expectations is not greater than
\begin{eqnarray*}
\lefteqn{\E\{e^{\beta\xi};\ \eta\le 0\}
-\E\{e^{\beta\xi(x)};\ -e^x<\eta\le 0\}}\\
&\le& \E\{e^{\beta\xi};\ \eta\le 0\}
-\E\{e^{\beta\xi}(1+\eta e^{-x})^\beta;\ -e^x<\eta\le 0\},
\end{eqnarray*}
due to \eqref{xi.x.below}.
If $\beta\le 1$ then $(1-z)^\beta\ge 1-z^\beta$ for all $z\in[0,1]$, so
\begin{eqnarray*}
\E\{e^{\beta\xi}(1+\eta e^{-x})^\beta;\ -e^x<\eta\le 0\}
&\ge& \E\{e^{\beta\xi}(1-|\eta|^\beta e^{-\beta x});\ -e^x<\eta\le 0\}\\
&\ge& \E\{e^{\beta\xi}(1-|\eta|^\beta e^{-\beta x});\ \eta\le 0\}.
\end{eqnarray*}
Therefore,
\begin{eqnarray*}
\E\{e^{\beta\xi};\ \eta\le 0\}
-\E\{e^{\beta\xi(x)};\ -e^x<\eta\le 0\}
&\le& e^{-\beta x}\E\{e^{\beta\xi}|\eta|^\beta;\ \eta\le 0\}\\
&=& O(e^{-\beta x}).
\end{eqnarray*}
If $\beta>1$ then $(1-z)^\beta\ge 1-\beta z$ for all $z\in[0,1]$, so
\begin{eqnarray*}
\E\{e^{\beta\xi}(1+\eta e^{-x})^\beta;\ -e^x<\eta\le 0\}
&\ge& \E\{e^{\beta\xi}(1-\beta|\eta| e^{-x});\ -e^x<\eta\le 0\}\\
&\ge& \E\{e^{\beta\xi}(1-\beta |\eta| e^{-x});\ \eta\le 0\}.
\end{eqnarray*}
Thus,
\begin{eqnarray*}
\E\{e^{\beta\xi};\ \eta\le 0\}
-\E\{e^{\beta\xi(x)};\ -e^x<\eta\le 0\}
&\le& \beta e^{-x}\E\{e^{\beta\xi}|\eta|;\ \eta\le 0\}\\
&=& O(e^{-x}),
\end{eqnarray*}
because both $e^{\beta\xi}$ and $e^{\beta\xi}|\eta|^\beta$
have finite expectations. Hence the third integral in
\eqref{decomp.3} is of order $O(e^{-\delta x})$ as $x\to\infty$.

Above bounds for integrals in \eqref{decomp.3} 
yield \eqref{cond.3.ver} and hence \eqref{D.infty.beta}.
In particular, $\E e^{\beta\xi(x)}\ge 1-O(e^{-\delta x})$ 
which implies $c>0$ by Theorem \ref{thm:m.c.cramer}.
\end{proof}

Now let us show how Theorem \ref{thm:m.c.cramer} allows 
to identify asymptotic tail behaviour of $D_n$.
Unfortunately it only works in the case of positive $\eta$, so $D_n>0$,
because the associated Markov chain $X_n$ does not satisfy
the condition \eqref{maj} if $X_n$ takes values of both signs;
in order to solve the case of general $\eta$ it is necessary 
to improve Theorem \ref{thm:m.c.cramer} but we do not do it in this paper.

\begin{Theorem}\label{thm:Cramer.n}
Let conditions of Theorem \ref{thm:Cramer} hold and $\eta>0$. 
If, in addition, $\sigma^2:=\E \xi^2e^{\beta\xi}-\alpha^2<\infty$,
\begin{eqnarray}\label{maj.1.0}
\E |\xi|e^{\beta\xi}\eta^\beta &<& \infty
\end{eqnarray}
and
\begin{eqnarray}\label{maj.1.2}
\E (\xi^++\log(1+\eta))^2 e^{\beta\xi^+}(1+\eta)^\beta &<& \infty,
\end{eqnarray}
then
$$
\P\{D_n>x\}=\frac{c}{x^\beta}
{\mathcal N}_{0,\sigma^2}\bigg(\frac{n\alpha-\log x}
{\sqrt{\alpha^{-1}\log x}}\bigg)
+o(1/x^\beta)
$$
as $x\to\infty$ uniformly in $n\ge0$.
In particular, for $n>\alpha^{-1}\log x+U(x)\sqrt{\log x}$
where $U(x)\to\infty$,
$$
\P\{D_n>x\}\sim\P\{D_\infty>x\}\quad\mbox{as }x\to\infty.
$$

The same results hold for $\widetilde D_n$ if we replace 
the conditions \eqref{maj.1.0} and \eqref{maj.1.2} by integrability 
of $|\xi|\eta^\beta$ and $\eta^\beta\log^2(1+\eta)$.
\end{Theorem}

\begin{proof}
The associated Markov chain $X_n$ satisfies all the conditions 
of Theorem \ref{thm:m.c.cramer}. Indeed, the condition \eqref{cond.1} 
is fulfilled because $\eta>0$.
The condition \eqref{maj} is satisfied with 
$\overline\xi=\xi^++\log(1+\eta)$, by \eqref{maj.1.2}.

Let us now prove an equivalent version of the condition \eqref{cond.2},
\begin{eqnarray}\label{cond.2.ver}
\E \xi(x)e^{\beta\xi(x)} &=& \E \xi e^{\beta\xi}+o(1/\sqrt x)
\quad\mbox{ as }x\to\infty.
\end{eqnarray}
Indeed, owing to $D_n>0$ and the definition of $\xi(x)$ we have 
\begin{eqnarray*}
\E \xi(x)e^{\beta\xi(x)} &=& \E \bigl[\log(e^\xi(e^x+\eta))-x\bigr]
e^{\beta[\log(e^\xi(e^x+\eta))-x]}+O(xe^{-\beta x})\\
&=& \E(\xi+\log(1+\eta e^{-x}))e^{\beta\xi}(1+\eta e^{-x})^\beta+O(xe^{-\beta x}).
\end{eqnarray*}
Therefore, due to $\eta>0$,
\begin{eqnarray}\label{est.rc}
\lefteqn{|\E \xi(x)e^{\beta\xi(x)}-\E \xi e^{\beta\xi}|}\nonumber\\
&=& \bigl|\E(\xi+\log(1+\eta e^{-x}))e^{\beta\xi}(1+\eta e^{-x})^\beta
-\E \xi e^{\beta\xi}\bigr|+O(xe^{-\beta x})\nonumber\\
&\le& \E |\xi|e^{\beta\xi}[(1+\eta e^{-x})^\beta-1]
+\E \log(1+\eta e^{-x})e^{\beta\xi}(1+\eta e^{-x})^\beta+O(xe^{-\beta x}).\nonumber\\[-1mm]
\end{eqnarray}
Condition \eqref{maj.1.0} allows us to repeat calculations 
used in the proof of Theorem \ref{thm:Cramer} and to show that
\begin{eqnarray}\label{est.rc.1}
\E |\xi|e^{\beta\xi}[(1+\eta e^{-x})^\beta-1]
&=& O(e^{-\delta x})\quad\mbox{ as }x\to\infty,
\end{eqnarray}
where again $\delta:=\min(\beta,1)>0$.
The second term on the right side of \eqref{est.rc} 
may be estimated from above as follows: 
\begin{eqnarray}\label{est.rc.2}
\lefteqn{\E \log(1+\eta e^{-x})e^{\beta\xi}(1+\eta e^{-x})^\beta}\nonumber\\
&\le& \E \{\log(1+\eta e^{-x})e^{\beta\xi}(1+\eta)^\beta;
\eta\le e^{x/2}\}\nonumber\\
&&\hspace{40mm} +\E \{\log(1+\eta)e^{\beta\xi}(1+\eta)^\beta;
\eta>e^{x/2}\}\nonumber\\
&\le& O(e^{-x/2})+\E \Bigl\{\frac{\log^2(1+\eta)}{\log(1+e^{x/2})}
e^{\beta\xi}(1+\eta)^\beta;\eta>e^{x/2}\Bigr\}\nonumber\\
&=& o(1/x)\quad\mbox{ as }x\to\infty,
\end{eqnarray}
due to \eqref{maj.1.2}.
Substituting \eqref{est.rc.1} and \eqref{est.rc.2}
into \eqref{est.rc} we justify \eqref{cond.2.ver}
and the result for $D_n$ follows.

In order to prove asymptotics for $\widetilde D_n$, 
we first notice that inequality
\begin{eqnarray*}
(e^{\xi^+}+\eta)^\beta \log^2(e^{\xi^+}+\eta) &\le& 
(2e^{\xi^+})^\beta \log^2(2e^{\xi^+})+(2\eta)^\beta \log^2(1+2\eta),
\end{eqnarray*}
together with conditions $\E \xi^2 e^{\beta\xi}<\infty$ and 
$\E\eta^\beta\log^2(1+\eta)<\infty$ implies that
\begin{eqnarray*}
\E (e^{\xi^+}+\eta)^\beta \log^2(e^{\xi^+}+\eta) &<& \infty.
\end{eqnarray*}
This observation makes it possible to conclude the proof 
for $\widetilde D_n$ similarly to $D_n$.
\end{proof}

The last theorem is proven in the case $\eta>0$ only.
As mentioned above, in the case where $\eta$ takes negative values 
the condition \eqref{maj} may fail for $\xi(x)$ for $x<0$.
However the same proving arguments allow us to prove 
a conditional central limit theorem for
\begin{eqnarray*}
T_x &:=& \min\{n\ge 1:D_n> x\}
\end{eqnarray*}
without assumption that $\eta$ is positive.

\begin{Theorem}\label{thm:T}
Suppose that $\E e^{\beta\xi}=1$, $\alpha:=\E \xi e^{\beta\xi}<\infty$
and $\sigma^2:=\E \xi^2e^{\beta\xi}-\alpha^2<\infty$. 
Suppose also that $\E e^{\beta\xi}|\eta|^\beta<\infty$,
\begin{eqnarray}\label{maj.1.0.abs}
\E |\xi|e^{\beta\xi}|\eta|^\beta &<& \infty
\end{eqnarray}
and
\begin{eqnarray}\label{maj.1.2.abs}
\E (\xi^++\log(1+\eta^+))^2 e^{\beta\xi^+}(1+|\eta|)^\beta &<& \infty,
\end{eqnarray}
then
\begin{eqnarray*}
\P\{T_x\le n\mid T_x<\infty\} &=& 
{\mathcal N}_{0,\sigma^2}\bigg(\frac{n\alpha-\log x}
{\sqrt{\alpha^{-1}\log x}}\bigg) +o(1)
\end{eqnarray*}
as $x\to\infty$ uniformly in $n\ge0$.
In particular, for $n>\alpha^{-1}\log x+U(x)\sqrt{\log x}$
where $U(x)\to\infty$,
\begin{eqnarray*}
\P\{T_x\le n\mid T_x<\infty\} &\to& 1\quad\mbox{as }x\to\infty.
\end{eqnarray*}
\end{Theorem}

Large deviation estimates for exceedance times of
$\widetilde D_n$ are studied by Buraczewski et al. in \cite{BCDZ}.
In particular, Theorem 2.2 of that paper states the same result
as the last theorem but under essentially stronger moment 
assumptions on $\xi$ and $\eta$.

\begin{proof}
Firstly,
\begin{eqnarray*}
\P\{T_x\le n\mid T_x<\infty\} &=& \P\{M_n>x\mid M_\infty>x\}
\quad\mbox{where }M_n:=\max_{k\le n} D_k.
\end{eqnarray*}
Secondly, well known duality says that $M_n$ equals in distribution 
to $M_n^*$ defined by the recursion
\begin{eqnarray*}
M_n^* &=& e^{\xi_n}(M_{n-1}^*+\eta_n)^+,
\end{eqnarray*}
because
\begin{eqnarray*}
M_n^* &=& \max\bigl(0,e^{\xi_n}\eta_n,
e^{\xi_n}\eta_n+e^{\xi_n+\xi_{n-1}}\eta_{n-1},\ldots\bigr).
\end{eqnarray*}
The sequence $M_n^*$ is a nonnegative Markov chain. 
Consider the associated Markov chain $X_n^* := f(M_n^*)$---where 
the function $f$ is defined in \eqref{eq:f}---whose jumps are
\begin{eqnarray*}
\xi^*(x) &=& f(e^\xi(f^{-1}(x)+\eta)^+)-x\\
&=& f(e^\xi(e^x+\eta)^+)-x\quad\mbox{if }x>1\\
&=& \xi+\log(1+\eta e^{-x})\quad\mbox{if }x>1\mbox{ and }e^\xi(e^x+\eta)>e.
\end{eqnarray*}
Then, for $x>1$,
\begin{eqnarray*}
\P\{T_x\le n\mid T_x<\infty\} &=& 
\frac{\P\{X_n^*>\log x\}}{\P\{M_\infty>x\}},
\end{eqnarray*}
so it suffices to prove that the Markov chain $X_n^*$
satisfies all the conditions of Theorem \ref{thm:m.c.cramer}.
The same arguments as in the proof of Theorem \ref{thm:Cramer}
show that $X_n^*$ satisfies the condition \eqref{cond.3}.
The majorisation condition \eqref{maj} holds with
$\overline\xi=\xi^++\log(1+\eta^+)$ as in the proof 
of Theorem \ref{thm:Cramer.n} because $X_n^*$ is positive.
The condition \eqref{cond.1} is also clear. 

Concerning condition \eqref{cond.2}, let us notice that the proof 
of \eqref{cond.2.ver} only use positivity of $\eta$ in \eqref{est.rc.2}.
So, it remains to show that
\begin{eqnarray*}
\E \{|\log(1+\eta e^{-x})|e^{\beta\xi}(1+\eta e^{-x})^\beta;\ -e^x<\eta<0\}
&=& o(1/\sqrt x)\quad\mbox{ as }x\to\infty.
\end{eqnarray*}
Indeed, this expectation does not exceed the sum
\begin{eqnarray*}
\E \{|\log(1+\eta e^{-x})|e^{\beta\xi};\ -e^{x/2}<\eta<0\}
+c\E \{e^{\beta\xi};\ \eta\le -e^{x/2}\}
\end{eqnarray*}
where $c=\sup_{-1<t<0}(1+t)^\beta|\log(1+t)|<\infty$.
The first expectation is of order $O(e^{-x/2})$
while the second is not greater than
\begin{eqnarray*}
\E \Bigl\{\frac{(1+|\eta|)^\beta}{(1+e^{x/2})^\beta}
e^{\beta\xi};\ \eta>e^{x/2}\Bigr\}
&=& O(e^{-\beta x/2})\quad\mbox{ as }x\to\infty,
\end{eqnarray*}
because $\E(1+|\eta|)^\beta e^{\beta\xi}<\infty$ 
by the condition \eqref{maj.1.2.abs} and the proof is complete.
\end{proof}

\section{Markov modulated perpetuities}
\label{sec:Markov}

In this section we consider a Markov modulated perpetuity. 
Let $\Phi_n$ be a time homogeneous non-periodic
Markov chain in a general state-space ${\textsf X}$.
Assume this chain possesses a Harris recurrent atom $x_*\in{\textsf X}$ 
(the case of general Markov chain with splitting may also be considered;
see Tweedie \cite{Tweedie} for precise definitions.)
For simplicity, let $\Phi_0=x_*$.
Let $(\xi,\eta)$, $(\xi_1,\eta_1)$, $(\xi_2,\eta_2)$
be independent identically distributed random vectors
in $\R^2$ independent of the process $\Phi$.
The Markov modulated perpetuity takes the form
\begin{eqnarray*}
D_\infty &:=& \sum_{n=1}^\infty g(\Phi_n,\eta_n)e^{S_n}
\end{eqnarray*}
where the sum $S_n$ is defined as $S_0=0$ and
$S_{n+1}=S_n+f(\Phi_{n+1},\xi_{n+1})$.
The functions $f$, $g:{\textsf X}\times\R\to\R$ 
are assumed to be deterministic. 

The Markov chain $\Phi$ is assumed to be positive
recurrent with invariant measure $\rho$. 
Then the Harris Markov chain $(\Phi_n,\xi_n)$ in
${\textsf X}\times\R$ has invariant measure
$\rho\times F$ where $F$ is the distribution of $\xi$.
We suppose that the mean drift of $S_n$ in stationary
regime of $\Phi$ is negative, that is,
\begin{eqnarray}\label{neg.m}
a: &=& \int_{\textsf X}\E f(x,\xi)\rho(dx)\nonumber\\
&=& \int_{{\textsf X}\times\R}f(x,y)(\rho\times F)(dx,dy)
< 0.
\end{eqnarray}
Then by the strong law of large numbers for the Markov
chain $(\Phi_n,\xi_n)$ (see Tweedie \cite[Chapter 17]{Tweedie}),
$S_n/n\to -a<0$ a.s. as $n\to\infty$ and $D_\infty$
is finite with probability 1 provided that, for instance,
the family of random variables $\log(1+|g(x,\eta)|)$,
$x\in{\textsf X}$, possesses an integrable majorant.

Denote $\tau:=\min\{n\ge 1:\Phi_n=x_*\mid \Phi_0=x_*\}$,
so that $\pi(x_*)=1/\E\tau$.
Notice that under positive recurrence $D_\infty$
possesses the following representation based on the
cycle structure of the underlying Markov chain $\Phi$
$$
D_\infty=\widehat\eta_1 e^{\widehat S_1}
+\widehat\eta_2 e^{\widehat S_2}
+\widehat\eta_3 e^{\widehat S_3}+...
$$
where $\widehat\eta_n$ are independent random variables
distributed as
$$
\widehat\eta_1:=\sum_{k=1}^\tau g(\Phi_k,\eta_k)
e^{S_k-S_\tau}
=\sum_{k=1}^\tau g(\Phi_k,\eta_k)
e^{-(f(\Phi_{k+1},\xi_{k+1})+\ldots+f(\Phi_\tau,\xi_\tau))}
$$
and $\widehat S_n$ is the sum of other independent
identically distributed random variables distributed
as $\widehat S_1=\widehat\xi_1:=S_\tau$ given $\Phi_0=x_*$.
Then the problem of approximation of the probability
$\P\{D_\infty>x\}$ as $x\to\infty$ can be reduced to
a perpetuity with independent vectors.
Let us demonstrate this under the Cram\'er setting. The function
$$
\widehat\varphi(\lambda):=\int_{\textsf X}\E e^{\lambda f(x,\xi)}\rho(dx)
$$
is convex, $\widehat\varphi(0)=1$ and $\widehat\varphi'(0)=a<0$.
Assume that there exists a $\beta>0$ such that $\widehat\varphi(\beta)=1$.
Then $\E e^{\beta\widehat\xi}=\E e^{\beta S_\tau}=1$.
If, in addition,
\begin{eqnarray}\label{maj.mm.0}
\E \widehat\xi e^{\beta\widehat\xi}
&=& \E S_\tau e^{\beta S_\tau}\ <\ \infty
\end{eqnarray}
and
\begin{eqnarray}\label{maj.mm.1}
\E e^{\beta\widehat\xi}|\widehat\eta|^\beta
&=& \E \Bigl|\sum_{k=1}^\tau g(\Phi_k,\eta_k)e^{S_k}\Bigr|^\beta
\ <\ \infty,
\end{eqnarray}
then by Theorem \ref{thm:Cramer} applied to $(\widehat\xi,\widehat\eta)$
we get for Markov modulated perpetuity that
\begin{eqnarray*}
\P\{D_\infty>x\} &\sim& \frac{c}{x^\beta}\quad\mbox{ as }x\to\infty.
\end{eqnarray*}
In terms of local characteristics, the conditions \eqref{maj.mm.0} 
and \eqref{maj.mm.1} will be automatically fulfilled if, for example,
\begin{eqnarray*}
C_1 &:=& \sup_{x\in{\textsf X}} 
\E e^{\beta f(x,\xi)} |f(x,\xi)|\ <\ \infty,\\
C_2 &:=& \sup_{x\in{\textsf X}} 
\E e^{\beta f(x,\xi)} |g(x,\eta)|^\beta\ <\ \infty,\\
K &:=& \sup_{x\in{\textsf X}} \E e^{\beta f(x,\xi)}\ <\ \infty
\end{eqnarray*}
and $\E \tau^{\max(\beta,1)}K^\tau<\infty$;
sufficiency of these conditions is based on conditioning on the trajectory 
of $\Phi_k$ and further application of the inequality
\begin{eqnarray*}
|x_1+\ldots+x_n|^\beta &\le& n^{\max(\beta-1,0)}(|x_1|^\beta+\ldots+|x_n|^\beta).
\end{eqnarray*}

This approach for proving power asymtotics for Markov modulated perpetuity 
is a simple alternative to how it is done by de Saporta in \cite{S2005} 
and in more general setting by Blanchet et al. in \cite[Theorem 1]{BLZ}
via Perron--Frobenius theorem which requires
a Markov chain $\Phi_n$ to be finite and calls for 
excessive exponential moment conditions on $f(x,\xi)$
and excessive power moment conditions on $g(x,\eta)$.

\section{Limit theorems for transient perpetuities}
\label{sec:clt}

Let $(\xi,\eta)$, $(\xi_1,\eta_1)$, $(\xi_2,\eta_2)$,
\ldots\ be independent identically distributed random
vectors valued in $\R\times\R^+$ were $\xi$'s have common 
positive mean $a>0$ and $\E \log(1+\eta)<\infty$.
By the strong law of large numbers, $S_n/n\to a>0$ as
$n\to\infty$ with probability $1$, so that, for every
fixed $\varepsilon>0$ there exists a.s. finite $N=N(\omega)$
such that $(a-\varepsilon)n\le S_n(\omega)\le (a+\varepsilon)n$
and $\log(1+\eta_n)\le\varepsilon n$ for $n\ge N$. Therefore,
$$
\frac{\log D_n}{n}
=\frac{\log\sum_{k=1}^n \eta_ke^{S_k}}{n}
\stackrel{a.s.}\to a\quad\mbox{ as }n\to\infty.
$$
The weak convergence for $D_n$ is specified in the following theorem.

\begin{Theorem}\label{thm:clt}
Suppose that $a:=\E \xi>0$,
$\sigma^2:={\mathbb Var}\xi<\infty$ and $\E \log^2(1+\eta)<\infty$.
Then the following weak convergence holds:
\begin{eqnarray}\label{unst.normal}
\frac{\log D_n-an}{\sqrt{n\sigma^2}} &\Rightarrow& {\mathcal N}_{0,1}
\quad\mbox{as }n\to\infty.
\end{eqnarray}
If further the distribution of $\xi$ is nonlattice,
then for any fixed $\Delta>0$ 
\begin{eqnarray}\label{unst.normal.local}
\P\{\log D_n\in(x,x+\Delta]\} &=& \frac{\Delta}{\sqrt{2\pi\sigma^2 n}}
e^{-(x-na)^2/\sqrt{2n\sigma^2}} +o(1/\sqrt n)
\end{eqnarray}
as $n\to\infty$ uniformly in $x$.

The same results hold for $\widetilde D_n$.
\end{Theorem}

Various limit behaviour of unstable perpetuities including
convergence \eqref{unst.normal} determined by the central limit theorem 
were studied in Rachev and Samorodnitsky \cite{RS}, 
Hitczenko and Weso\l owski \cite{HW}; Markov modulated perpetuities were
considered in Basu and Roitershtein \cite{BR}. 
Notice that the standard technique---simpler than 
presented here---used in these papers for proving limit theorems 
does not allow to derive a local result like \eqref{unst.normal.local}.

Our proof of Theorem \ref{thm:clt} is based on the following result.

\begin{Theorem}[\cite{K2001}]\label{thm:m.c.clt}
Suppose that the Markov chain $X_n$ on $\R$ 
goes to infinity almost surely. Let, for some $a>0$ and $\sigma^2>0$,
\begin{eqnarray}\label{clt.1}
\E \xi(x) &=& a+o(1/\sqrt x),\\
{\mathbb Var}\xi(x) &\to& \sigma^2
\quad\mbox{ as }x\to\infty,\label{clt.2}
\end{eqnarray}
and the family $\{\xi^2(x),x>1\}$ of squares of the
jumps is integrable uniformly in $x$. Then
\begin{eqnarray}\label{ccl.m.c}
\frac{X_n-an}{\sqrt{n\sigma^2}} &\Rightarrow& {\mathcal N}_{0,1}
\quad\mbox{ as }n\to\infty.
\end{eqnarray}
Let, in addition, the Markov chain $X_n$ be asymptotically 
homogeneous in space, that is $\xi(x)\Rightarrow\xi$ as $x\to\infty$, 
the distribution of $\xi$ be nonlattice, for any $A>0$
\begin{eqnarray}\label{char.gen}
\sup_{x>na/2,\ \lambda\in[-A,A]}
\bigl|\E e^{i\lambda\xi(x)}-\E e^{i\lambda\xi}\bigr| &=& o(1/n)
\end{eqnarray}
as $n\to\infty$ and
\begin{eqnarray}\label{c..n.gen}
\P\bigl\{X_k\le ka/2\ \mbox{ for some }\ k\ge n\bigr\} &=& o(1/\sqrt n).
\end{eqnarray}
Then for any fixed $\Delta>0$ 
\begin{eqnarray}\label{ccl.m.c.local}
\P\{X_n\in(x,x+\Delta]\} &=& \frac{\Delta}{\sqrt{2\pi\sigma^2 n}}
e^{-(x-na)^2/\sqrt{2n\sigma^2}} +o(1/\sqrt n)
\end{eqnarray}
as $n\to\infty$ uniformly in $x$.
\end{Theorem}

\begin{proof}[Proof of Theorem \ref{thm:clt}]
It is sufficient to check that the Markov chain $X_n$ associated to $D_n$
satisfies all the conditions of Theorem \ref{thm:m.c.clt}.

To prove that \eqref{clt.1} is satisfied, first note that $\eta>0$
implies $\xi(x)\ge\xi+\log(1+\eta e^{-x})\ge\xi$. In addition,
\begin{eqnarray*}
\E \log(1+\eta e^{-x})
&=& \E \{\log(1+\eta e^{-x});\eta\le e^{x/2}\}
+\E \{\log(1+\eta e^{-x});\eta>e^{x/2}\}\\
&\le& e^{-x/2}+\E\Bigl\{\frac{\log^2(1+\eta)}{\log(1+e^{x/2})};
\eta>e^{x/2}\Bigr\}\\
&=& e^{-x/2}+o(1/x)\quad\mbox{ as }x\to\infty,
\end{eqnarray*}
by the condition $\E \log^2(1+\eta)<\infty$ 
which is even better than \eqref{clt.1}.

Since $\xi^2(x)\le 2\xi^2+2\log^2(1+\eta)$ for all $x>0$
and the mean of the right hand side is finite,
the family $\{\xi^2(x),x>0\}$
is integrable uniformly in $x$. In particular,
then the convergence \eqref{clt.2} follows and hence \eqref{unst.normal}.

Further, \eqref{char.gen} follows because
\begin{eqnarray*}
\bigl|\E e^{i\lambda\xi(x)}-\E e^{i\lambda\xi}\bigr| &=& 
\bigl|\E e^{i\lambda\xi}(1-e^{i\lambda\log(1+\eta e^{-x})})\bigr|\\
&\le& \E \bigl|1-e^{i\lambda\log(1+\eta e^{-x})}\bigr|\\
&\le& \sqrt{2\E(1-\cos(\lambda\log(1+\eta e^{-x})))}
\end{eqnarray*}
which implies for $|\lambda|\le A$ 
\begin{eqnarray*}
\lefteqn{\bigl|\E e^{i\lambda\xi(x)}-\E e^{i\lambda\xi}\bigr|}\\
&\le& 2\sqrt{\P\{\eta>e^{x/2}\}}+\sqrt{2(1-\E\{\cos(A\log(1+\eta e^{-x})));\ 
\eta\le e^{x/2}\}}\\
&=& o(1/x)+O(e^{-x/2}).
\end{eqnarray*}
Finally, \eqref{c..n.gen} is satisfied due to
\begin{eqnarray*}
\P\bigl\{X_k\le ka/2\mbox{ for some }k\ge n\bigr\} &\le& 
\P\bigl\{S_k\le ka/2\mbox{ for some }k\ge n\bigr\}\\
&\le& \P\Bigl\{\frac{S_k-ka}{k}\le -\frac{a}{2}\mbox{ for some }k\ge n\Bigr\},
\end{eqnarray*}
where the sequence $(S_k-ka)/k$, $k=n$, $n+1$, \ldots\ constitutes 
a reverse martingale which allows to apply Kolmogorov's inequality
for martingales:
\begin{eqnarray*}
\P\Bigl\{\frac{S_k-ka}{k}\le -\frac{a}{2}\mbox{ for some }k\ge n\Bigr\}
&\le& \P\Bigl\{\sup_{k\ge n}\Bigl|\frac{S_k-ka}{k}\Bigr| \ge \frac{a}{2}\Bigr\}\\
&\le& \frac{4{\mathbb Var}S_n}{a^2n^2}\\
&=& O(1/n)=o(1/\sqrt n).
\end{eqnarray*}
So the Markov chain $X_n$ associated to $D_n$ satisfies 
all the conditions of Theorem \ref{thm:m.c.clt} and the proof is complete.
\end{proof}

A local limit analysis of the Green function for asymptotically
homogeneous in space transient Markov chains is done in \cite{Korshunov2008}.
If $\E\xi>0$ and $\E\log(1+\eta)<\infty$, 
then Theorem 1 from that paper is applicable to the associated 
Markov chain $X_n$ and in the case where the distribution of $\xi$ 
is non-lattice we obtain that, for every fixed $h>0$,
$$
\sum_{n=1}^\infty \P\{\log D_n\in(x,x+h]\} \ \to\ \frac{h}{\E\xi}
\quad\mbox{as }x\to\infty.
$$ 

Let us also mention a link to the so-called Lamperti's problem
which is about the limit behaviour of a Markov chain with asymptotically zero 
drift, that is, when $\E\xi(x)\to 0$ as $x\to\infty$,
it goes back to Lamperti \cite{Lamperti}. 
It was proven by Denisov et al. \cite[Theorem 4]{DKW} 
that if we consider a Markov chain $X_n$ such that 
\begin{itemize}
\item[(i)] $\E\xi(x)\sim \mu/x$ and $\E\xi^2(x)\to b$ as $x\to\infty$, $\mu>-b/2$, 
\item[(ii)] the family $\{\xi^2(x), x\}$ possesses an integrable majorant,
\item[(iii)] $X_n\to\infty$ in probability as $n\to\infty$, 
\end{itemize}
then $X_n^2/n$ converges weakly to the $\Gamma$-distribution 
with mean $2\mu+b$ and variance $(2\mu+b)2b$. 

If we consider an unstable perpetuity with $\E\xi=0$, 
$\sigma^2:={\mathbb Var}\xi^2<\infty$ and $\E\log^2(1+\eta)<\infty$, 
then the associated Markov chain $X_n$ for $D_n$ 
satisfies the above conditions with $\mu=0$ and $b=\sigma^2$ and hence
is null-recurrent and 
\begin{eqnarray*}
\frac{\log^2 D_n}{n} &\Rightarrow& \Gamma_{1/2,\ 2\sigma^2}=N^2(0,\sigma^2)
\quad\mbox{as }n\to\infty,
\end{eqnarray*}
first proven by Hitczenko and Weso\l owski in \cite[Theorem 3(i)]{HW}.


\begin{thebibliography}{9}

\bibitem{Aldous}
Aldous, D. (1989)
{\it Probability Approximations Via the Poisson Clumping Heuristic.}
Springer, New York. 

\bibitem{BR}
Basu, R., Roitershtein, A. (2013)
Divergent perpetuities modulated by regime switches.
{\it Stoch. Models} {\bf 29}, 129--148.

\bibitem{BLZ}
Blanchet, J., Lam, H., Zwart, B. (2012)
Efficient rare-event simulation for perpetuities.
{\it Stochastic Process. Appl.}, {\bf 122}, 3361--3392.

\bibitem{BK2001}
Borovkov, A. A., Korshunov, D. (2001)
Large-deviation probabilities for one-dimensional Markov chains.
Part 2: Prestationary distributions in the exponential case.
{\it Theory Probab. Appl.} {\bf 45}, 379--405.

\bibitem{BK2002}
Borovkov, A. A., Korshunov, D. (2002)
Large-deviation probabilities for one-dimensional Markov chains.
Part 3: Prestationary distributions in the subexponential case.
{\it Theory Probab. Appl.} {\bf 46}, 603--618.

\bibitem{BCDZ}
Buraczewski, D., Collamore, J. F., Damek, E., Zienkiewicz, J.
Large deviation estimates for exceedance times of
perpetuity sequences and their dual processes.
{\it Ann. Probab.} to appear.

\bibitem{DKW}
Denisov, D., Korshunov, D., Wachtel, V. (2013)
Potential analysis for positive recurrent Markov chains
with asymptotically zero drift: Power-type asymptotics.
{\it Stochastic Process. Appl.} {\bf 123}, 3027--3051.

\bibitem{Dyszewski}
Dyszewski, P. (2016)
Iterated random functions and slowly varying tails.
{\it Stochastic Process. Appl.} {\bf 126}, 392--413.

\bibitem{EG}
Embrechts, P. and Goldie, C.M. (1994). 
Perpetuities and random equations. 
In {\it Asymptotic Statistics (Prague, 1993). Contrib. Statist.},
75–-86. Heidelberg: Physica.

\bibitem{Feller}
Feller, W. (1971)
{\it An Introduction to Probability Theory and Its
Applications} Vol. 2, Wiley, New York.

\bibitem{FKZ}
Foss, S., Korshunov, D., Zachary, S. (2011)
{\it An Introduction to Heavy-Tailed and Subexponential Distributions}.
Springer, New York.

\bibitem{Goldie1991}
Goldie, C. M. (1991)
Implicit renewal theory and tails of solutions
of random equations.
{\it Ann. Appl. Probab.} {\bf 1}, 126--166.

\bibitem{GG}
Goldie, C. M., Gr\"ubel, R. (1996)
Perpetuities with thin tails. 
{\it Adv. Appl. Probab.} {\bf 28}, 463--480.

\bibitem{GoldieMaller}
Goldie, C. M., Maller, R. A. (2000)
Stability of perpetuities.
{\it Ann. Probab.} {\bf 28}, 1195--1218.

\bibitem{Grey}
Grey, D. (1994)
Regular variation in the tail behaviour 
of solutions of random difference equations.
{\it Ann. Appl. Probab.} {\bf 4}, 169--183.

\bibitem{HW}
Hitczenko, P., Weso\l owski, J. (2011)
Renorming divergent perpetuities. 
{\it Bernoulli} {\bf 17}, 880--894.

\bibitem{Kesten1973}
Kesten, H. (1973)
Random difference equations and renewal theory 
for products of random matrices.
{\it Acta Math.} {\bf 131}, 207--248.

\bibitem{KM}
Konstantinides, D.G., Mikosch, T. (2005) 
Large deviations and ruin probabilities for solutions 
to stochastic recurrence equations with heavy-tailed innovations. 
{\it Ann. Probab.} {\bf 33}, 1992--2035.

\bibitem{K2001}
Korshunov, D. (2001)
Limit theorems for general Markov chains.
{\it Sib. Math. J.} {\bf 42}, 301--316.


\bibitem{K2004}
Korshunov, D. (2004)
One-dimensional asymptotically homogeneous Markov chains:
Cram\'er transform and large deviation probabilities.
{\it Siberian Adv. Math.} {\bf 14}, N 4, 30--70.

\bibitem{Korshunov2008}
Korshunov, D. (2008)
The key renewal theorem for a transient Markov chain.
{\it J. Theoret. Probab.} {\bf 21}, 234--245.

\bibitem{Lamperti}
Lamperti, J. (1962)
A new class of probability limit theorems. 
{\it J. Math. Mech.} {\bf 11}, 749--772.

\bibitem{Tweedie}
Meyn, S., Tweedie, R. (2009)
{\it Markov Chains and Stochastic Stability}, 2nd Ed.,
Cambridge Univ. Press.


\bibitem{RS}
Rachev, S.T., Samorodnitsky, G. (1995)
Limit laws for a stochastic process and random recursion arising
in probabilistic modeling. 
{\it Adv. Appl. Probab.} {\bf 27}, 185--202.

\bibitem{S2005}
de Saporta, B. (2005)
Tail of the stationary solution of the stochastic
equation $Y_{n+1}=a_nY_n+b_n$ with Markovian coefficients.
{\it Stochastic Process. Appl.} {\bf 115}, 1954--1978.

\bibitem{Vervaat}
Vervaat, W. (1979)
On a stochastic difference equation and a representation 
of nonnegative infinitely divisible random variables. 
{\it Adv. Appl. Probab.} {\bf 11}, 750--783.

\end{thebibliography}
\end{document}